\documentclass[a4paper,reqno]{amsart}
\usepackage{eurosym}
\usepackage{enumerate,amsmath,amsthm,amsfonts}
\usepackage{amsmath}
\usepackage{amsfonts}
\usepackage{amssymb}
\usepackage{epsfig}
\usepackage{graphicx,color}
\usepackage{graphpap,color}
\usepackage{graphicx}
\usepackage{latexsym}
\usepackage{float}
\usepackage{setspace}

\newtheorem{theorem} {Theorem}[section]

\newtheorem{definition}[theorem] {Definition}
\newtheorem{example}[theorem] {Example}
\newtheorem{corollary}[theorem] {Corollary}

.360pk
\font\rt=cmss9.360pk
\font\sd=cmcsc9.360pk
\include {mak}
\input{tcilatex}

\begin{document}

~\vspace{-16mm}



\oddsidemargin 16.5truemm \evensidemargin 16.5truemm

\thispagestyle{plain}

\vspace{-0.25cc}

\hspace{7.5cm}{\scriptsize Vol. xx, No. xx (20xx), pp.~xx--xx.}\newline
\vspace{1.2cc}

\vspace{1.5cc}

\begin{center}
{\Large \textbf{MEASURE OF NONCOMPACTNESS IN THE STUDY OF SOLUTIONS FOR A
SYSTEM OF INTEGRAL EQUATIONS\\[0pt]
\vspace{1.5cc} \textsc{Vatan Karakaya}{\large \textsc{$^{1}$, Mohammad
Mursaleen$^{2}$, Nour El Houda Bouzara$^{3}$, and Derya Sekman $^{4}$}}\\[0pt%
]
}}

\bigskip

{\Large \textbf{{\small $^{1}$Department of Mathematical Engineering, Yildiz
Technical University, Davutpasa Campus, Esenler, 34210 Istanbul, Turkey,
vkkaya@yahoo.com }}}

{\Large \textbf{{\small $^{2}$Department of Mathematics, Aligarh Muslim
University, Aligarh 202002, India, mursaleenm@gmail.com }}}

{\Large \textbf{{\small $^{3}$ Faculty of Mathematics, University of Science
and Technology Houari Boumedi\`{e}ne, Bab-Ezzouar, 16111 Algies, Algeria,
bzr.nour@gmail.com }}}

{\Large \textbf{{\small $^{4}$ Faculty of Arts and Sciences, Department of
Mathematics, Ahi Evran University, 40100 Kirsehir, Turkey,
deryasekman@gmail.com}}}
\end{center}

{\Large \textbf{\rule{0mm}{6mm}\renewcommand{\thefootnote}{}\footnotetext{%
{\scriptsize \textit{2010 Mathematics Subject Classification}:47H10, 47H08,
45G15 \newline
\textrm{Received: dd-mm-yyyy, accepted: dd-mm-yyyy.} }} }}

\begin{center}
\vspace{1.5cc}

\parbox{24cc}{{\Small{\bf Abstract.}
In this work, we prove the existence of solutions for a tripled system of integral equations using some new results of fixed point theory associated with measure of noncompactness. These results extend the results in some previous works. Also, the condition under which the operator admits fixed points is more general than the others in literature.
}}
\end{center}

\vspace{0.25cc} 
\parbox{24cc}{\Small {\it Key words and Phrases}: Measure of
noncompactness, Fixed point, System of integral equations.}

\vspace{1.5cc}

\section{INTRODUCTION}


In recent years, measure of noncompactness which was given by Kuratowski 
\cite{kuratovski} and has provided powerful tools for obtaining the
solutions of a large variety of integral equations and systems. One can find
related references in studies involving Aghajani et al. \cite{Aghajani11}, 
\cite{Aghajanicoupled}, \cite{Aghajani2}, \cite{ek1}, Banas \cite{banas1},
Banas and Rzepka \cite{banas3}, Mursaleen and Mohiuddine \cite{murs},
Mursaleen and Rizvi \cite{ek3}, Araba et al. \cite{reza}, Deepmala and \
Pathak \cite{deepmala}, Shaochun and Gan \cite{Shaochun}, Sikorska \cite%
{sikorska}, Alotaibi et al. \cite{ek2}, and many others.

In this paper, we have studied solvability of the system given by%
\begin{equation*}
\left\{ 
\begin{array}{c}
x\left( t\right) =g_{1}\left( t\right) +f_{1}\left( t,x\left( \theta
_{1}\left( t\right) \right) ,y\left( \theta _{1}\left( t\right) \right)
,z\left( \theta _{1}\left( t\right) \right) ,\varphi \left(
\int_{0}^{r_{1}\left( t\right) }h\left( t,s,x\left( \nu _{1}\left( s\right)
\right) ,y\left( \nu _{1}\left( s\right) \right) ,z\left( \nu _{1}\left(
s\right) \right) \right) ds\right) \right)  \\ 
y\left( t\right) =g_{2}\left( t\right) +f_{2}\left( t,x\left( \theta
_{2}\left( t\right) \right) ,y\left( \theta _{2}\left( t\right) \right)
,z\left( \theta _{2}\left( t\right) \right) ,\varphi \left(
\int_{0}^{r_{2}\left( t\right) }h\left( t,s,x\left( \nu _{2}\left( s\right)
\right) ,y\left( \nu _{2}\left( s\right) \right) ,z\left( \nu _{2}\left(
s\right) \right) \right) ds\right) \right)  \\ 
z\left( t\right) =g_{3}\left( t\right) +f_{3}\left( t,x\left( \theta
_{3}\left( t\right) \right) ,y\left( \theta _{3}\left( t\right) \right)
,z\left( \theta _{3}\left( t\right) \right) ,\varphi \left(
\int_{0}^{r_{3}\left( t\right) }h\left( t,s,x\left( \nu _{3}\left( s\right)
\right) ,y\left( \nu _{3}\left( s\right) \right) ,z\left( \nu _{3}\left(
s\right) \right) \right) ds\right) \right) 
\end{array}%
\right. ,
\end{equation*}%
by establishing some results of existence for fixed points of condensing
operators in Banach spaces.

Throughout this paper,we assume tha $X$ is a Banach space. Also we denote $%
\mathcal{B}_{X}$, $\overline{X}$ and $ConvX,$ the family of bounded subset,
closure and closed convex hull of $X$ respectively.

We now gather some well-known definitions and results from the literature
which will be used throughout this paper.

\begin{definition}[\protect\cite{banas2}]
Let $X$ be a Banach space and $\mathcal{B}_{X}$ the family of bounded subset
of $X.$ A map%
\begin{equation*}
\tau :\mathcal{B}_{X}\rightarrow \left[ 0,\infty \right)
\end{equation*}%
which satisfies the following:

\begin{enumerate}
\item $\tau \left( A\right) =0\Leftrightarrow A$ is a precompact set,

\item $A\subset B\Rightarrow \tau \left( A\right) \leqslant \tau \left(
B\right) ,$

\item $\tau \left( A\right) =\tau \left( \overline{A}\right) ,$ $\forall
A\in \mathcal{B}_{X},$

\item $\tau \left( ConvA\right) =\tau \left( A\right) ,$

\item $\tau \left( \lambda A+\left( 1-\lambda \right) B\right) \leqslant
\lambda \tau \left( A\right) +\left( 1-\lambda \right) \tau \left( B\right)
, $ for $\lambda \in \left[ 0,1\right] ,$

\item Let $\left( A_{n}\right) $ be a sequence of closed sets from $\mathcal{%
B}_{X}$\ such that $A_{n+1}\subseteq A_{n},$ $\left( n\geqslant 1\right) $
and $\lim\limits_{n\rightarrow \infty }\mu \left( A_{n}\right) =0.$ Then,
the set $A_{\infty }=\underset{n=1}{\overset{\infty }{\cap }}A_{n}$ is
nonempty and $A_{\infty }$ is precompact.
\end{enumerate}
\end{definition}

The functional $\tau $\ is called measure of noncompactness defined on the
Banach space $X$.

\begin{theorem}[\protect\cite{lec.note}]
Let $A$ be a nonempty closed, bounded and convex subset of $X$. If $%
T:A\rightarrow A$ is a continuous mapping on the subset $C\subset A$ 
\begin{equation*}
\tau \left( TC\right) \leqslant k\tau \left( C\right) ,\text{ \ }k\in \left[
0,1\right) \text{,}
\end{equation*}%
then $T$ has a fixed point.
\end{theorem}

The following theorem is considered as a generalization of Darbo fixed point
theorem.

\begin{theorem}[\protect\cite{mursaleen}]
Let $A\ $be a nonempty closed, bounded and convex subset of $X$ and $%
T:A\rightarrow A$ be a continuous mapping for any subset $C\subset A$%
\begin{equation*}
\tau \left( TC\right) \leqslant \gamma \left( \tau \left( C\right) \right)
\tau \left( C\right) ,
\end{equation*}%
where $\gamma :\mathbb{R}_{+}\rightarrow \left[ 0,1\right) $ that is $\gamma
\left( t_{n}\right) \rightarrow 1$ implies $t_{n}\rightarrow 0.$ Then, $T$
has at least one fixed point.
\end{theorem}

\begin{corollary}[\protect\cite{mursaleen}]
\label{fpth}Let $A\ $\ be a nonempty closed, bounded and convex subset of $X$
and $T:A\rightarrow A$ be a continuous mapping for any subset $C\subset A$%
\begin{equation*}
\tau \left( TC\right) \leqslant \varphi \left( \tau \left( C\right) \right) ,
\end{equation*}%
where $\varphi :\mathbb{R}_{+}\rightarrow \mathbb{R}_{+}$ is a nondecreasing
and upper semicontinuous functions, that is, for every $t>0,$ $\varphi
\left( t\right) <t.$ Then, $T$ has at least one fixed point.
\end{corollary}

\begin{theorem}[\protect\cite{kamenski}]
\label{th copy}Let $\tau _{1},$ $\tau _{2},$ ...,$\tau _{n}$ be measures of
noncompactness in Banach spaces $A_{1},$ $A_{2},$ ...$A_{n},$ $\left( \text{%
respectively}\right) $.

Then the function%
\begin{equation*}
\widetilde{\tau }\left( X\right) =F\left( \tau _{1}\left( X_{1}\right) ,\tau
_{2}\left( X_{2}\right) ,...,\tau _{n}\left( X_{n}\right) \right) ,
\end{equation*}%
defines a measure of noncompactness in $A_{1}\times A_{2}\times $ ...$\times
A_{n}$ where $X_{i}$ is the natural projection of $X$\ on $A_{i},$ for $%
i=1,2,...,n,$ and $F$ be a convex function defined by%
\begin{equation*}
F:\left[ 0,\infty \right) \times \left[ 0,\infty \right) \times ...\times %
\left[ 0,\infty \right) \rightarrow \left[ 0,\infty \right) ,
\end{equation*}%
such that,%
\begin{equation*}
F\left( x_{1},x_{2},...,x_{n}\right) =0\Leftrightarrow x_{i}=0,\text{ for }%
i=1,2,...,n.
\end{equation*}
\end{theorem}

\begin{example}[\protect\cite{nour}]
We can notice that by taking 
\begin{equation*}
F\left( x,y,z\right) =\max \left\{ x,y,z\right\} \text{ for any }\left(
x,y,z\right) \in \left[ 0,\infty \right) \times \left[ 0,\infty \right)
\times \left[ 0,\infty \right) ,
\end{equation*}%
or%
\begin{equation*}
F\left( x,y,z\right) =x+y+z\text{ for any }\left( x,y,z\right) \in \left[
0,\infty \right) \times \left[ 0,\infty \right) \times \left[ 0,\infty
\right) .
\end{equation*}%
Then, $F$ satisfies the conditions of Theorem \ref{th copy}. Thus, for a
measure of noncompactness $\tau _{i}$ $\left( i=1,2,3\right) $, we have that 
\begin{equation*}
\widetilde{\tau }\left( X\right) =\max \left( \tau _{1}\left( X_{1}\right)
,\tau _{2}\left( X_{2}\right) ,\tau _{3}\left( X_{3}\right) \right) ,
\end{equation*}%
or 
\begin{equation*}
\widetilde{\tau }\left( X\right) =\tau _{1}\left( X_{1}\right) +\tau
_{2}\left( X_{2}\right) +\tau _{3}\left( X_{3}\right) ,
\end{equation*}%
defines a measure of noncompactness in the space $A\times A\times A$ where $%
X_{i}$, $i=1,2,3$ are the natural projections of $X$ on $A_{i}$.
\end{example}

\section{MAIN RESULTS}

\begin{theorem}
\label{mainth}Let $A$ be a nonempty, bounded ,closed and convex subset of a
Banach space $X$\ and let $\varphi :\mathbb{R}^{+}\rightarrow \mathbb{R}^{+}$
be a nondecreasing and upper semicontinous function such that $\varphi
\left( t\right) <t$ for all $t>0.$ Then for any measure of noncompactness $%
\tau $, and continuous operators $T_{i}:A\times A\times A\rightarrow A$ $%
\left( i=1,2,3\right) $ satisfying%
\begin{equation}
\tau \left( T_{i}\left( X_{1}\times X_{2}\times X_{3}\right) \right)
\leqslant \varphi \left( \max \left( \tau \left( X_{1}\right) ,\tau \left(
X_{2}\right) ,\tau \left( X_{3}\right) \right) \right) ,\text{ }%
X_{1},X_{2},X_{3}\in A,  \label{cndt}
\end{equation}%
there exist $u,v,z\in A$ such that$\left\{ 
\begin{array}{c}
T_{1}\left( u,v,z\right) =u \\ 
T_{2}\left( u,v,z\right) =v \\ 
T_{3}\left( u,v,z\right) =z%
\end{array}%
\right. .$
\end{theorem}

\textsc{Proof. }Consider the following measure of noncompactness%
\begin{equation*}
\widetilde{\tau }\left( A\times A\times A\right) =\max \left( \tau \left(
X_{1}\right) ,\tau \left( X_{2}\right) ,\tau \left( X_{3}\right) \right) ,
\end{equation*}%
where $X_{1},X_{2},X_{3}\in A$ and the mapping $\ T:A\times A\times
A\rightarrow A,$ 
\begin{equation*}
T\left( u,v,z\right) =\left( T_{1}\left( u,v,z\right) ,T_{2}\left(
u,v,z\right) ,T_{3}\left( u,v,z\right) \right) .
\end{equation*}%
We have,%
\begin{eqnarray*}
\widetilde{\tau }\left( T\left( A\times A\times A\right) \right) &=&%
\widetilde{\tau }\left( \left( T_{1}\left( X_{1}\times X_{2}\times
X_{3}\right) \right) ,T_{2}\left( X_{1}\times X_{2}\times X_{3}\right)
,T_{3}\left( X_{1}\times X_{2}\times X_{3}\right) \right) \\
&=&\max \left\{ \tau \left( T_{1}\left( X_{1}\times X_{2}\times X_{3}\right)
\right) ,\tau \left( T_{2}\left( X_{1}\times X_{2}\times X_{3}\right)
\right) ,\tau \left( T_{3}\left( X_{1}\times X_{2}\times X_{3}\right)
\right) \right\} \\
&\leqslant &\max \left\{ \varphi \left( \max \left( \tau \left( X_{1}\right)
,\tau \left( X_{2}\right) ,\tau \left( X_{3}\right) \right) \right) ,\varphi
\left( \max \left( \tau \left( X_{1}\right) ,\tau \left( X_{2}\right) ,\tau
\left( X_{3}\right) \right) \right) ,\right. \\
&&\left. \varphi \left( \max \left( \tau \left( X_{1}\right) ,\tau \left(
X_{2}\right) ,\tau \left( X_{3}\right) \right) \right) \right\} .
\end{eqnarray*}%
By hypothesis $\varphi $\ is a non-decreasing function, then%
\begin{eqnarray*}
\widetilde{\mu }\left( T\left( A\times A\times A\right) \right) &\leqslant
&\varphi \left[ \max \left\{ \max \left( \mu \left( X_{1}\right) ,\mu \left(
X_{2}\right) ,\mu \left( X_{3}\right) \right) ,\max \left( \mu \left(
X_{1}\right) ,\mu \left( X_{2}\right) ,\mu \left( X_{3}\right) \right)
\right. \right. , \\
&&\left. \left. \max \left( \mu \left( X_{1}\right) ,\mu \left( X_{2}\right)
,\mu \left( X_{3}\right) \right) \right\} \right] .
\end{eqnarray*}%
Consequently,%
\begin{equation*}
\widetilde{\mu }\left( T\left( A\times A\times A\right) \right) \leqslant
\varphi \left( \widetilde{\mu }\left( A\times A\times A\right) \right) .
\end{equation*}%
So,%
\begin{equation*}
\mu \left( T_{1}\left( x,y,z\right) ,T_{2}\left( x,y,z\right) ,T_{3}\left(
x,y,z\right) \right) \leqslant \varphi \left( \max \left( \mu \left(
X_{1}\right) ,\mu \left( X_{2}\right) ,\mu \left( X_{3}\right) \right)
\right) .
\end{equation*}%
By Corollary \ref{fpth}, we conclude that there exist $x^{\ast },y^{\ast
},z^{\ast }\in A$ such that%
\begin{equation*}
T\left( x^{\ast },y^{\ast },z^{\ast }\right) =\left( x^{\ast },y^{\ast
},z^{\ast }\right) .
\end{equation*}%
In the other hand,%
\begin{equation*}
T\left( x^{\ast },y^{\ast },z^{\ast }\right) =\left( T_{1}\left( x^{\ast
},y^{\ast },z^{\ast }\right) ,T_{2}\left( x^{\ast },y^{\ast },z^{\ast
}\right) ,T_{3}\left( x^{\ast },y^{\ast },z^{\ast }\right) \right) .
\end{equation*}%
Hence,%
\begin{equation*}
\left\{ 
\begin{array}{c}
T_{1}\left( x^{\ast },y^{\ast },z^{\ast }\right) =x^{\ast } \\ 
T_{2}\left( x^{\ast },y^{\ast },z^{\ast }\right) =y^{\ast } \\ 
T_{3}\left( x^{\ast },y^{\ast },z^{\ast }\right) =z^{\ast }%
\end{array}%
\right. .
\end{equation*}

\begin{definition}[\protect\cite{nour}]
A tripled $\left( x,y,z\right) $ of a mapping $T:A\times A\times
A\rightarrow A,$ is called a tripled fixed point if 
\begin{equation*}
T\left( x,y,z\right) =x,\text{ \ }T\left( y,x,z\right) =y\text{ \ and \ }%
T\left( z,y,x\right) =z.
\end{equation*}
\end{definition}

\begin{remark}
Let $T:A\times A\times A\rightarrow A$ be a continuous mapping. If we define 
$T_{1}\left( x,y,z\right) =T\left( x,y,z\right) ,$ $T_{2}\left( x,y,z\right)
=T\left( y,x,z\right) $ and $T_{3}\left( x,y,z\right) =T\left( z,y,x\right) ,
$ then main results of \cite{nour} can be considered as a result of Theorem %
\ref{mainth}.
\end{remark}

It is very natural to extend the above result from three dimensions to
multidimensional fixed point and in the same way we can prove the following
theorem.

\begin{theorem}
Let $A$ be a nonempty, bounded ,closed and convex subset of a Banach space $%
X $\ and let $\varphi :\mathbb{R}^{+}\rightarrow \mathbb{R}^{+}$ be a
nondecreasing and upper semicontinous function such that $\varphi \left(
t\right) <t$ for all $t>0.$ Then for any measure of noncompactness $\mu \ $%
and for continuous operators $T_{i}:A^{n}\rightarrow \Omega $ \ $\left(
i=1,...,n\right) $ satifying%
\begin{equation*}
\mu \left( T_{i}\left( X_{1}\times ...\times X_{n}\right) \right) \leqslant
\varphi \left( \max \left( \mu \left( X_{1}\right) ,...,\mu \left(
X_{n}\right) \right) \right) ,\text{ }X_{i}\in A,\text{ }i=\overline{1,n},
\end{equation*}%
there exist $x_{1}^{\ast },...,x_{n}^{\ast }$ such that%
\begin{equation*}
\left\{ 
\begin{array}{c}
T_{1}\left( x_{1}^{\ast },...,x_{n}^{\ast }\right) =x_{1}^{\ast } \\ 
\vdots \\ 
T_{n}\left( x_{1}^{\ast },...,x_{n}^{\ast }\right) =x_{n}^{\ast }%
\end{array}%
\right. .
\end{equation*}
\end{theorem}

As a particular case we get the following corollary:

\begin{corollary}[\protect\cite{Aghajanicoupled}]
Let $A$ be a nonempty, bounded ,closed and convex subset of a Banach space $%
X $\ and let $\varphi :\mathbb{R}^{+}\rightarrow \mathbb{R}^{+}$ be a
nondecreasing and upper semicontinous function such that $\varphi \left(
t\right) <t$ for all $t>0.$ Then for any measure of noncompactness $\mu ,$
the continuous operator $G:A^{n}\rightarrow A$ satisfying%
\begin{equation*}
\mu \left( G\left( X_{1}\times ...\times X_{n}\right) \right) \leqslant
k\max \left( \mu \left( X_{1}\right) ,...,\mu \left( X_{n}\right) \right) ,%
\text{ }X_{1},...,X_{n}\in A.
\end{equation*}
\end{corollary}

And for the case $n=2,$ we have the following result.

\begin{corollary}[\protect\cite{Aghajanicoupled}]
Let $A$ be a nonempty, bounded ,closed and convex subset of a Banach space $%
X $\ and let $\varphi :\mathbb{R}^{+}\rightarrow \mathbb{R}^{+}$ be a
nondecreasing and upper semicontinous function such that $\varphi \left(
t\right) <t$ for all $t>0.$ Then for any measure of noncompactness $\mu ,$
the continuous operator $G:A\times A\rightarrow A$ satisfying%
\begin{equation*}
\mu \left( G\left( X_{1}\times X_{2}\right) \right) \leqslant k\max \left(
\mu \left( X_{1}\right) ,\mu \left( X_{2}\right) \right) ,\text{ }%
X_{1},X_{2}\in A.
\end{equation*}
\end{corollary}

In the following we choose for the space $X$ the space $BC\left( \mathbb{R}%
^{+}\right) $, i.e., the space of all real functions defined, bounded and
continuous on $\mathbb{R}^{+}.$ Then, we get the following theorem.

\begin{theorem}
\label{main}Let $A$ be a nonempty, bounded, closed and convex subset of $%
BC\left( \mathbb{R}^{+}\right) $ and $T_{i}:A\times A\times A\rightarrow A$
be a continuous operator such that for every $x,y,z,u,v,w\in A.$%
\begin{equation}
\left\Vert T_{i}\left( x,y,z\right) -T_{i}\left( u,v,w\right) \right\Vert
_{\infty }\leqslant \varphi \left( \max \left\{ \left\Vert x-u\right\Vert
_{\infty },\left\Vert y-v\right\Vert _{\infty },\left\Vert z-w\right\Vert
_{\infty }\right\} \right) ,  \label{cndtn}
\end{equation}%
where $\varphi :\mathbb{R}^{+}\rightarrow \mathbb{R}^{+}$ is a nondecreasing
and upper semicontinuous function such that $\varphi \left( t\right) <t$ for
all $t>0$. Then there exist $x^{\ast },y^{\ast },z^{\ast }\in A$ such that,$%
\left\{ 
\begin{array}{c}
T_{1}\left( x^{\ast },y^{\ast },z^{\ast }\right) =x^{\ast } \\ 
T_{2}\left( x^{\ast },y^{\ast },z^{\ast }\right) =y^{\ast } \\ 
T_{3}\left( x^{\ast },y^{\ast },z^{\ast }\right) =z^{\ast }%
\end{array}%
\right. .$
\end{theorem}

\textsc{Proof. }To verify that the operator $T_{i}:A\times A\times
A\rightarrow A$ satisfy the condition $\left( \text{\ref{cndt}}\right) $ we
recall the following notions.

The measure of noncompactness on $BC\left( \mathbb{R}^{+}\right) $ for a
positive fixed $t$ on $\mathcal{B}_{BC\left( \mathbb{R}^{+}\right) }$\ is
defined as follows:%
\begin{equation*}
\mu \left( X\right) =\omega _{0}\left( X\right) +\lim \sup_{t\rightarrow
\infty }diamX\left( t\right) ,
\end{equation*}%
that is, $diamX\left( t\right) =\sup \left\{ \left\vert x\left( t\right)
-y\left( t\right) \right\vert :x,y\in X\right\} ,$ $X\left( t\right)
=\left\{ x\left( t\right) :x\in X\right\} .$and 
\begin{equation*}
\omega _{0}\left( X\right) =\lim_{K\rightarrow \infty }\omega _{0}^{K}\left(
X\right) ,
\end{equation*}%
\begin{equation*}
\omega _{0}^{K}\left( X\right) =\lim_{\epsilon \rightarrow 0}\omega
^{K}\left( X,\epsilon \right) ,
\end{equation*}%
\begin{equation*}
\omega ^{K}\left( X,\epsilon \right) =\sup \left\{ \omega ^{K}\left(
x,\epsilon \right) :x\in X\right\} ,
\end{equation*}%
\begin{equation*}
\omega ^{K}\left( x,\epsilon \right) =\sup \left\{ \left\vert x\left(
t\right) -x\left( s\right) \right\vert :t,s\in \left[ 0,K\right] ,\text{ }%
\left\vert t-s\right\vert \leqslant \epsilon \right\} \text{, for }K>0,
\end{equation*}%
where $\omega ^{K}\left( x,\epsilon \right) $ for $x\in X$ and $\epsilon >0,$
is the modulus of continuity of $x$ on the compact $\left[ 0,K\right] $,
where $K$ is a positive number.

We have%
\begin{equation*}
\left\Vert T_{i}\left( x,y,z\right) \left( t\right) -T_{i}\left(
x,y,z\right) \left( s\right) \right\Vert \leqslant \varphi \left( \max
\left\{ \left\Vert x\left( t\right) -x\left( s\right) \right\Vert
,\left\Vert y\left( t\right) -y\left( s\right) \right\Vert ,\left\Vert
z\left( t\right) -z\left( s\right) \right\Vert \right\} \right) ,
\end{equation*}%
by taking the supremum and using the fact that $\varphi $ is nondecreasing,
we get%
\begin{equation*}
\omega ^{K}\left( T_{i}\left( x,y,z\right) ,\epsilon \right) \leqslant
\varphi \left( \max \left\{ \omega ^{K}\left( x,\epsilon \right) ,\omega
^{K}\left( y,\epsilon \right) ,\omega ^{K}\left( z,\epsilon \right) \right\}
\right) .
\end{equation*}%
Thus,%
\begin{equation}
\omega _{0}\left( T_{i}\left( X_{1}\times X_{2}\times X_{3}\right) \right)
\leqslant \varphi \left( \max \left\{ \omega _{0}\left( X_{1}\right) ,\omega
_{0}\left( X_{2}\right) ,\omega _{0}\left( X_{3}\right) \right\} \right) .
\label{modulos}
\end{equation}%
Since in $\left( \text{\ref{cndtn}}\right) $ $x,y$ and $z$ are arbitrary and 
$\varphi $ is non-decreasing,%
\begin{equation*}
DiamT_{i}\left( X_{1}\times X_{2}\times X_{3}\right) \left( t\right)
\leqslant \varphi \left( \max \left\{ DiamX_{1}\left( t\right) ,\text{ }%
DiamX_{2}\left( t\right) ,DiamX_{3}\left( t\right) \right\} \right) .
\end{equation*}%
In further, $X_{1}\left( t\right) ,X_{2}\left( t\right) ,X_{3}\left(
t\right) $ are subspaces of $BC\left( \mathbb{R}_{+}\right) $. Then,%
\begin{multline*}
\lim \sup_{t\rightarrow \infty }DiamT_{i}\left( X_{1}\times X_{2}\times
X_{3}\right) \left( t\right) \leqslant \lim \sup_{t\rightarrow \infty
}\varphi \left( \max \left\{ DiamX_{1}\left( t\right) ,\text{ }%
DiamX_{2}\left( t\right) ,DiamX_{3}\left( t\right) \right\} \right) +\Phi
\left( \epsilon \right) \\
\leqslant \varphi \left( \max \left\{ \lim \sup_{t\rightarrow \infty
}DiamX_{1}\left( t\right) ,\lim \sup_{t\rightarrow \infty }\text{ }%
DiamX_{2}\left( t\right) ,\lim \sup_{t\rightarrow \infty }DiamX_{3}\left(
t\right) \right\} \right) .
\end{multline*}%
Using $\varphi \left( t\right) <t$ for all $t>0$ and from $\left( \text{\ref%
{modulos}}\right) $ and the above inequality, we get%
\begin{equation*}
\mu \left( T_{i}\left( X_{1}\times X_{2}\times X_{3}\right) \right)
\leqslant \varphi \left( \max \left( \mu \left( X_{1}\right) ,\mu \left(
X_{2}\right) ,\mu \left( X_{3}\right) \right) \right) ,\text{ }%
X_{1},X_{2},X_{3}\in A.
\end{equation*}%
Consequently, there exist $x^{\ast },y^{\ast },z^{\ast }\in A$ such that%
\begin{eqnarray*}
T\left( x^{\ast },y^{\ast },z^{\ast }\right) &=&\left( T_{1}\left( x^{\ast
},y^{\ast },z^{\ast }\right) ,T_{2}\left( x^{\ast },y^{\ast },z^{\ast
}\right) ,T_{3}\left( x^{\ast },y^{\ast },z^{\ast }\right) \right) \\
&=&\left( x^{\ast },y^{\ast },z^{\ast }\right) .
\end{eqnarray*}%
Thus, 
\begin{equation*}
\left\{ 
\begin{array}{c}
T_{1}\left( x^{\ast },y^{\ast },z^{\ast }\right) =x^{\ast } \\ 
T_{2}\left( x^{\ast },y^{\ast },z^{\ast }\right) =y^{\ast } \\ 
T_{3}\left( x^{\ast },y^{\ast },z^{\ast }\right) =z^{\ast }%
\end{array}%
\right. .
\end{equation*}

\section{APPLICATION}

Now, we will use the results of the previous section to resolve the
following system%
\begin{equation}
\left\{ 
\begin{array}{c}
x\left( t\right) =g_{1}\left( t\right) +f_{1}\left( t,x\left( \theta
_{1}\left( t\right) \right) ,y\left( \theta _{1}\left( t\right) \right)
,z\left( \theta _{1}\left( t\right) \right) ,\varphi \left(
\int_{0}^{q_{1}\left( t\right) }h\left( t,s,x\left( \eta _{1}\left( s\right)
\right) ,y\left( \eta _{1}\left( s\right) \right) ,z\left( \eta _{1}\left(
s\right) \right) \right) ds\right) \right) \\ 
y\left( t\right) =g_{2}\left( t\right) +f_{2}\left( t,x\left( \theta
_{2}\left( t\right) \right) ,y\left( \theta _{2}\left( t\right) \right)
,z\left( \theta _{2}\left( t\right) \right) ,\varphi \left(
\int_{0}^{q_{2}\left( t\right) }h\left( t,s,x\left( \eta _{2}\left( s\right)
\right) ,y\left( \eta _{2}\left( s\right) \right) ,z\left( \eta _{2}\left(
s\right) \right) \right) ds\right) \right) \\ 
z\left( t\right) =g_{3}\left( t\right) +f_{3}\left( t,x\left( \theta
_{3}\left( t\right) \right) ,y\left( \theta _{3}\left( t\right) \right)
,z\left( \theta _{3}\left( t\right) \right) ,\varphi \left(
\int_{0}^{q_{3}\left( t\right) }h\left( t,s,x\left( \eta _{3}\left( s\right)
\right) ,y\left( \eta _{3}\left( s\right) \right) ,z\left( \eta _{3}\left(
s\right) \right) \right) ds\right) \right)%
\end{array}%
\right. ,  \label{sys}
\end{equation}

\bigskip

\qquad We study system $\left( \text{\ref{sys}}\right) $ under the following
assumptions:

\begin{enumerate}
\item[$\left( i\right) $] $\xi _{i},\eta _{i},q_{i}:\mathbb{R}%
_{+}\rightarrow \mathbb{R}_{+},$ $\left( i=1,2,3\right) ,$ are continuous
and $\xi _{i}\left( t\right) \rightarrow \infty $ as $t\rightarrow \infty .$

\item[$\left( ii\right) $] The function $\psi _{i}:\mathbb{R\rightarrow R}$, 
$\left( i=1,2,3\right) ,$ is continuous and there exist positive $\delta
_{i},\alpha _{i}$\ such that 
\begin{equation*}
\left\vert \psi _{i}\left( t_{1}\right) -\psi _{i}\left( t_{2}\right)
\right\vert \leqslant \delta _{i}\left\vert t_{1}-t_{2}\right\vert ^{\alpha
_{i}},
\end{equation*}%
for $i=1,2,3$ and\ any $t_{1},t_{2}\in \mathbb{R}_{+}$.

\item[$\left( iii\right) $] $f_{i}:\mathbb{R}_{+}\times \mathbb{R\times
R\times R\times R\rightarrow R}$ are continuous, $g_{i}:\mathbb{R}_{+}%
\mathbb{\rightarrow R}$ are bounded and there exists nondecreasing
continuous function $\Phi _{i}:\mathbb{R}_{+}\mathbb{\rightarrow R}_{+}$
with $\Phi _{i}\left( 0\right) =0,$ $i=1,2,3,$ such that%
\begin{equation*}
\left\vert f_{i}\left( t,x_{1},x_{2},x_{3},x_{4}\right) -f_{i}\left(
t,y_{1},y_{2},y_{3},y_{4}\right) \right\vert \leqslant \left( \varphi
_{i}\left( \max \left\{ \left\vert x_{1}-y_{1}\right\vert ,\left\vert
x_{2}-y_{2}\right\vert ,\left\vert x_{3}-y_{3}\right\vert \right\} \right)
\right) +\Phi _{i}\left( \left\vert x_{4}-y_{4}\right\vert \right) .
\end{equation*}

\item[$\left( iv\right) $] The functions defined by $\left\vert f_{i}\left(
t,0,0,0,0\right) \right\vert ,$\ $i=1,2,3$ are bounded on $\mathbb{R}_{+}$,
i.e.,%
\begin{equation}
M_{i}=\sup \left\{ f_{i}\left( t,0,0,0,0\right) :t\in \mathbb{R}_{+}\right\}
<\infty .  \label{M}
\end{equation}

\item[$\left( v\right) $] $h_{i}:\mathbb{R}_{+}\mathbb{\times R}_{+}\mathbb{%
\times R\times R\times R\rightarrow R}$ , are continuous functions and there
exists a positive constant $D$ such that $i=1,2,3,$%
\begin{equation}
\sup \left\{ \left\vert \int_{0}^{q_{i}\left( t\right) }h_{i}\left(
t,s,x\left( \eta \left( s\right) \right) ,y\left( \eta \left( s\right)
\right) ,z\left( \eta \left( s\right) \right) \right) ds\right\vert :\text{ }%
t,s\in \mathbb{R}_{+},\text{ }x,y,z\in BC\left( \mathbb{R}_{+}\right)
\right\} <D,  \label{D}
\end{equation}%
and%
\begin{equation}
\lim_{t\rightarrow \infty }\int_{0}^{q_{i}\left( t\right) }\left[
h_{i}\left( t,s,x\left( \eta \left( s\right) \right) ,y\left( \eta \left(
s\right) \right) ,z\left( \eta \left( s\right) \right) \right) -h_{i}\left(
t,s,u\left( \eta \left( s\right) \right) ,v\left( \eta \left( s\right)
\right) ,w\left( \eta \left( s\right) \right) \right) \right] ds=0,
\label{limh}
\end{equation}%
with respect to $x,y,z,u,v,w\in BC\left( \mathbb{R}_{+}\right) .$
\end{enumerate}

Consider the following operator, 
\begin{equation*}
T_{i}\left( x,y,z\right) =g_{i}\left( t\right) +f_{i}\left( t,x\left( \xi
_{i}\left( t\right) \right) ,y\left( \xi _{i}\left( t\right) \right)
,z\left( \xi _{i}\left( t\right) \right) ,\psi \left( \int_{0}^{q_{i}\left(
t\right) }h\left( t,s,x\left( \eta _{i}\left( s\right) \right) ,y\left( \eta
_{i}\left( s\right) \right) ,z\left( \eta _{i}\left( s\right) \right)
\right) ds\right) \right) .
\end{equation*}%
Solving the system $\left( \text{\ref{sys}}\right) $ is equivalent to find
the fixed points of the operator $T_{i}$. Then let verify the conditions of
Theorem \ref{main}$.$

First, since $g_{i}$ and $f_{i}$\ $\left( i=1,2,3\right) $ are continuous
then the operators $T_{i}$ are continuous.

In further, for $x,y,z\in B_{r}$ $\left( r\right) $ $\left( \text{for }%
r>0\right) $ let,%
\begin{eqnarray*}
&&\left\Vert T_{i}\left( x,y,z\right) \left( t\right) \right\Vert \\
&=&\left\Vert g_{i}\left( t\right) +f_{i}\left( t,x\left( \xi _{i}\left(
t\right) \right) ,y\left( \xi _{i}\left( t\right) \right) ,z\left( \xi
_{i}\left( t\right) \right) ,\psi \left( \int_{0}^{q_{i}\left( t\right)
}h\left( t,s,x\left( \eta _{i}\left( s\right) \right) ,y\left( \eta
_{i}\left( s\right) \right) ,z\left( \eta _{i}\left( s\right) \right)
\right) ds\right) \right) \right\Vert \\
&\leqslant &\left\Vert f_{i}\left( t,x\left( \xi _{i}\left( t\right) \right)
,y\left( \xi _{i}\left( t\right) \right) ,z\left( \xi _{i}\left( t\right)
\right) ,\psi \left( \int_{0}^{q_{i}\left( t\right) }h\left( t,s,x\left(
\eta _{i}\left( s\right) \right) ,y\left( \eta _{i}\left( s\right) \right)
,z\left( \eta _{i}\left( s\right) \right) \right) ds\right) \right) \right.
\\
&&\left. -f\left( t,0,0,0\right) +f\left( t,0,0,0\right) \right\Vert
+\left\Vert g_{i}\left( t\right) \right\Vert \\
&\leqslant &\left\Vert g_{i}\left( t\right) \right\Vert +\left\Vert f\left(
t,0,0,0\right) \right\Vert \\
&&+\varphi _{i}\left( \max \left\{ \left\vert x\left( \xi _{i}\left(
t\right) \right) \right\vert ,\left\vert y\left( \xi _{i}\left( t\right)
\right) \right\vert ,\left\vert z\left( \xi _{i}\left( t\right) \right)
\right\vert \right\} \right) \\
&&+\Phi _{i}\left( \psi \left( \int_{0}^{q_{i}\left( t\right) }h\left(
t,s,x\left( \eta _{i}\left( s\right) \right) ,y\left( \eta _{i}\left(
s\right) \right) ,z\left( \eta _{i}\left( s\right) \right) \right) ds\right)
\right) .
\end{eqnarray*}%
Since, $g_{i}$ are bounded, $f_{i}$ are continuous functions and using
hypothesis $\left( iv\right) $-$\left( v\right) ,$ we get%
\begin{eqnarray*}
\left\Vert T_{i}\left( x,y,z\right) \right\Vert _{\infty } &\leqslant
&\varphi _{i}\left( \max \left\{ \left\Vert x\right\Vert _{\infty
},\left\Vert y\right\Vert _{\infty },\left\Vert z\right\Vert _{\infty
}\right\} \right) +G+M_{i}+\Phi _{i}\left( \delta _{i}D^{\alpha _{i}}\right)
\\
&\leqslant &\varphi _{i}\left( r\right) +G+M_{i}+\Phi _{i}\left( \delta
_{i}D^{\alpha _{i}}\right) ,
\end{eqnarray*}%
for some $r_{0}\geqslant 0,$ we obtain $T_{i}\left( B_{r_{0}}\times
B_{r_{0}}\times B_{r_{0}}\right) \subset B_{r_{0}}.$

Moreover,

\begin{eqnarray*}
&&\left\Vert T_{i}\left( x,y,z\right) -T_{i}\left( u,v,w\right) \right\Vert
_{\infty } \\
&=&\sup_{t}\left\Vert g_{i}\left( t\right) +f_{i}\left( t,x\left( \xi
_{i}\left( t\right) \right) ,y\left( \xi _{i}\left( t\right) \right)
,z\left( \xi _{i}\left( t\right) \right) ,\psi \left( \int_{0}^{q_{i}\left(
t\right) }h\left( t,s,x\left( \eta _{i}\left( s\right) \right) ,y\left( \eta
_{i}\left( s\right) \right) ,z\left( \eta _{i}\left( s\right) \right)
\right) ds\right) \right) \right. \\
&&\left. -g_{i}\left( t\right) -f_{i}\left( t,u\left( \xi _{i}\left(
t\right) \right) ,v\left( \xi _{i}\left( t\right) \right) ,w\left( \xi
_{i}\left( t\right) \right) ,\psi \left( \int_{0}^{q_{i}\left( t\right)
}h\left( t,s,u\left( \eta _{i}\left( s\right) \right) ,v\left( \eta
_{i}\left( s\right) \right) ,w\left( \eta _{i}\left( s\right) \right)
\right) ds\right) \right) \right\Vert \\
&=&\sup_{t}\left\Vert f_{i}\left( t,x\left( \xi _{i}\left( t\right) \right)
,y\left( \xi _{i}\left( t\right) \right) ,z\left( \xi _{i}\left( t\right)
\right) ,\psi \left( \int_{0}^{q_{i}\left( t\right) }h\left( t,s,x\left(
\eta _{i}\left( s\right) \right) ,y\left( \eta _{i}\left( s\right) \right)
,z\left( \eta _{i}\left( s\right) \right) \right) ds\right) \right) \right.
\\
&&\left. -f_{i}\left( t,u\left( \xi _{i}\left( t\right) \right) ,v\left( \xi
_{i}\left( t\right) \right) ,w\left( \xi _{i}\left( t\right) \right) ,\psi
\left( \int_{0}^{q_{i}\left( t\right) }h\left( t,s,u\left( \eta _{i}\left(
s\right) \right) ,v\left( \eta _{i}\left( s\right) \right) ,w\left( \eta
_{i}\left( s\right) \right) \right) ds\right) \right) \right\Vert \\
&\leqslant &\sup_{t}\left\{ \varphi _{i}\left( \max \left\{ \left\vert
x\left( \xi _{i}\left( t\right) \right) -u\left( \xi _{i}\left( t\right)
\right) \right\vert ,\left\vert y\left( \xi _{i}\left( t\right) \right)
-v\left( \xi _{i}\left( t\right) \right) \right\vert ,\left\vert z\left( \xi
_{i}\left( t\right) \right) -w\left( \xi _{i}\left( t\right) \right)
\right\vert \right\} \right) \right. \\
&&\left. +\Phi _{i}\left( \left\vert 
\begin{array}{c}
\psi \left( \int_{0}^{q_{i}\left( t\right) }h\left( t,s,x\left( \eta
_{i}\left( s\right) \right) ,y\left( \eta _{i}\left( s\right) \right)
,z\left( \eta _{i}\left( s\right) \right) \right) ds\right) \\ 
-\psi \left( \int_{0}^{q_{i}\left( t\right) }h\left( t,s,u\left( \eta
_{i}\left( s\right) \right) ,v\left( \eta _{i}\left( s\right) \right)
,w\left( \eta _{i}\left( s\right) \right) \right) ds\right)%
\end{array}%
\right\vert \right) \right\} \\
&\leqslant &\varphi _{i}\left( \max \left\{ \left\Vert x-u\right\Vert
_{\infty },\left\Vert y-v\right\Vert _{\infty },\left\Vert z-w\right\Vert
_{\infty }\right\} \right) \\
&&+\sup_{t}\Phi _{i}\left( \delta _{i}\left\vert \int_{0}^{q_{i}\left(
t\right) }\left\{ 
\begin{array}{c}
h\left( t,s,x\left( \eta _{i}\left( s\right) \right) ,y\left( \eta
_{i}\left( s\right) \right) ,z\left( \eta _{i}\left( s\right) \right) \right)
\\ 
-h\left( t,s,u\left( \eta _{i}\left( s\right) \right) ,v\left( \eta
_{i}\left( s\right) \right) ,w\left( \eta _{i}\left( s\right) \right) \right)%
\end{array}%
\right\} ds\right\vert ^{\alpha _{i}}\right) .
\end{eqnarray*}%
Consider,%
\begin{equation*}
\left\vert \int_{0}^{q_{i}\left( t\right) }\left\{ h\left( t,s,x\left( \eta
_{i}\left( s\right) \right) ,y\left( \eta _{i}\left( s\right) \right)
,z\left( \eta _{i}\left( s\right) \right) \right) -h\left( t,s,u\left( \eta
_{i}\left( s\right) \right) ,v\left( \eta _{i}\left( s\right) \right)
,w\left( \eta _{i}\left( s\right) \right) \right) \right\} ds\right\vert .
\end{equation*}%
Using the condition $\left( \text{\ref{limh}}\right) ,$ we get%
\begin{equation*}
\left\vert \int_{0}^{q_{i}\left( t\right) }\left\{ h\left( t,s,x\left( \eta
_{i}\left( s\right) \right) ,y\left( \eta _{i}\left( s\right) \right)
,z\left( \eta _{i}\left( s\right) \right) \right) -h\left( t,s,u\left( \eta
_{i}\left( s\right) \right) ,v\left( \eta _{i}\left( s\right) \right)
,w\left( \eta _{i}\left( s\right) \right) \right) \right\} ds\right\vert
\leqslant \epsilon
\end{equation*}%
and%
\begin{equation*}
\delta _{i}\left\vert \int_{0}^{q_{i}\left( t\right) }\left\{ h\left(
t,s,x\left( \eta _{i}\left( s\right) \right) ,y\left( \eta _{i}\left(
s\right) \right) ,z\left( \eta _{i}\left( s\right) \right) \right) -h\left(
t,s,u\left( \eta _{i}\left( s\right) \right) ,v\left( \eta _{i}\left(
s\right) \right) ,w\left( \eta _{i}\left( s\right) \right) \right) \right\}
ds\right\vert ^{\alpha _{i}}\leqslant \delta _{i}\epsilon ^{\alpha _{i}}.
\end{equation*}%
Thus,%
\begin{equation*}
\Phi _{i}\left( \delta _{i}\left\vert \int_{0}^{q_{i}\left( t\right)
}\left\{ 
\begin{array}{c}
h\left( t,s,x\left( \eta _{i}\left( s\right) \right) ,y\left( \eta
_{i}\left( s\right) \right) ,z\left( \eta _{i}\left( s\right) \right) \right)
\\ 
-h\left( t,s,u\left( \eta _{i}\left( s\right) \right) ,v\left( \eta
_{i}\left( s\right) \right) ,w\left( \eta _{i}\left( s\right) \right) \right)%
\end{array}%
\right\} ds\right\vert ^{\alpha _{i}}\right) \leqslant \Phi _{i}\left(
\delta _{i}\epsilon ^{\alpha _{i}}\right) .
\end{equation*}%
On the other hand $\Phi _{i}$ is continuous function and$\ \Phi _{i}\left(
0\right) =0$ , $\epsilon $ is arbitrary, then for $\epsilon \rightarrow 0,$
we get 
\begin{equation*}
\left\Vert T_{i}\left( x,y,z\right) -T_{i}\left( u,v,w\right) \right\Vert
_{\infty }\leqslant \varphi _{i}\left( \max \left\{ \left\Vert
x-u\right\Vert _{\infty },\left\Vert y-v\right\Vert _{\infty },\left\Vert
z-w\right\Vert _{\infty }\right\} \right) .
\end{equation*}%
Consequently by Theorem \ref{main},\ there exist $x^{\ast },y^{\ast
},z^{\ast }$ such that%
\begin{equation*}
\left\{ 
\begin{array}{c}
T_{1}\left( x^{\ast },y^{\ast },z^{\ast }\right) =x^{\ast } \\ 
T_{2}\left( x^{\ast },y^{\ast },z^{\ast }\right) =y^{\ast } \\ 
T_{3}\left( x^{\ast },y^{\ast },z^{\ast }\right) =z^{\ast }%
\end{array}%
\right. .
\end{equation*}%
Then, we had proved the following theorem.

\begin{theorem}
Under the conditions $\left( i\right) -\left( v\right) $ the system of
integral equations $\left( \text{\ref{sys}}\right) $ has at least one
solution in the space $BC\left( \mathbb{R}_{+}\right) \times BC\left( 
\mathbb{R}_{+}\right) \times BC\left( \mathbb{R}_{+}\right) $.
\end{theorem}

\begin{example}
Let the system of integral equations%
\begin{equation*}
\left\{ 
\begin{array}{l}
x\left( t\right) =\frac{t^{2}}{2+2t^{4}}+\frac{x\left( \sqrt{t}\right)
+y\left( \sqrt{t}\right) +z\left( \sqrt{t}\right) }{3t^{2}+3}+\arctan
\int_{0}^{\sqrt{t}}\frac{x\left( s^{2}\right) s\left\vert \sin y\left(
s^{2}\right) \right\vert \left\vert \cos z\left( s^{2}\right) \right\vert }{%
e^{t}\left( 1+x^{2}\left( s^{2}\right) \right) \left( 1+\sin ^{2}y\left(
s^{2}\right) \right) \left( 1+\cos ^{2}z\left( s^{2}\right) \right) }ds \\ 
y\left( t\right) =\frac{1}{2}e^{-t^{2}}+\frac{t^{2}\left( x\left( t\right)
+y\left( t\right) +z\left( t\right) \right) }{3t^{4}+3}+\sin \int_{0}^{t}%
\frac{e^{s}y^{2}\left( s\right) \left( 1+\cos ^{2}x\left( s\right) \right)
\left( 1+\sin ^{2}z\left( s\right) \right) }{e^{t^{2}}\left( 1+y^{2}\left(
s\right) \right) \left( 1+\sin ^{2}x\left( s\right) \right) \left( 1+\cos
^{2}z\left( s\right) \right) }ds \\ 
z\left( t\right) =\frac{1}{2\sqrt{1+t^{4}}}+\frac{t^{3}\left( x\left(
t\right) +y\left( t\right) +z\left( t\right) \right) }{3t^{5}+3}+\cos
\int_{0}^{t^{2}}\frac{s^{2}\left\vert \cos z\left( s\right) \right\vert +%
\sqrt{e^{s}\left( 1+z^{2}\left( s\right) \right) \left( 1+\sin ^{2}y\left(
s\right) \right) \left( 1+\cos ^{2}x\left( s\right) \right) }}{e^{t}\left(
1+z^{2}\left( s\right) \right) \left( 1+\sin ^{2}y\left( s\right) \right)
\left( 1+\cos ^{2}x\left( s\right) \right) }ds%
\end{array}%
\right. .
\end{equation*}%
We notice that by taking%
\begin{equation*}
g_{1}\left( t\right) =\frac{t^{2}}{2+2t^{4}}\text{, }g_{2}\left( t\right) =%
\frac{1}{2}e^{-t^{2}},\text{ }g_{3}\left( t\right) =\frac{1}{2\sqrt{1+t^{4}}}%
,
\end{equation*}%
\begin{equation*}
\begin{array}{l}
f_{1}\left( t,x,y,z,p\right) =\frac{x+y+z}{3t^{2}+3}+p \\ 
f_{2}\left( t,x,y,z,p\right) =\frac{t^{2}\left( x+y+z\right) }{3t^{4}+3}+p
\\ 
f_{3}\left( t,x,y,z,p\right) =\frac{t^{3}\left( x+y+z\right) }{3t^{5}+3}+p%
\end{array}%
,
\end{equation*}%
\begin{eqnarray*}
h_{1}\left( t,s,x,y,z\right) &=&\frac{xs\left\vert \sin y\right\vert
\left\vert \cos z\right\vert }{e^{t}\left( 1+x^{2}\right) \left( 1+\sin
^{2}y\right) \left( 1+\cos ^{2}z\right) } \\
h_{2}\left( t,s,x,y,z\right) &=&\frac{e^{s}\left( 1+y^{2}\right) \left(
1+\sin ^{2}x\right) \left( 1+\cos ^{2}z\right) }{e^{t^{2}}\left(
1+y^{2}\right) \left( 1+\sin ^{2}x\right) \left( 1+\cos ^{2}z\right) } \\
h_{3}\left( t,s,x,y,z\right) &=&\frac{s^{2}\left\vert \cos z\right\vert +%
\sqrt{e^{s}\left( 1+z^{2}\right) \left( 1+\sin ^{2}y\right) \left( 1+\cos
^{2}x\right) }}{e^{t}\left( 1+z^{2}\right) \left( 1+\sin ^{2}y\right) \left(
1+\cos ^{2}x\right) }
\end{eqnarray*}%
and%
\begin{eqnarray*}
\eta _{1}\left( t\right) &=&t^{2},\text{ }\eta _{2}\left( t\right) =\eta
_{3}\left( t\right) =t \\
\xi _{1}\left( t\right) &=&\sqrt{t},\text{ }\xi _{2}\left( t\right) =\xi
_{3}\left( t\right) =t \\
q_{1}\left( t\right) &=&\sqrt{t},\text{ }q_{2}\left( t\right) =t,\text{ }%
q_{3}\left( t\right) =t^{2} \\
\Psi _{1}\left( t\right) &=&\arctan t,\text{ }\Psi _{2}\left( t\right) =\sin
t,\text{ }\Psi _{3}\left( t\right) =\cos t,
\end{eqnarray*}%
we get the system of integral equations $\left( \text{\ref{sys}}\right) .$

To solve this system we need to verify the conditions $\left( i\right)
-\left( v\right) $.

Obviously, $\xi _{i},\eta _{i},q_{i}:\mathbb{R}_{+}\rightarrow \mathbb{R}%
_{+} $ are continuous and $\xi ^{i}\rightarrow \infty $ as $t\rightarrow
\infty .$ In further, the functions $\psi _{i}:\mathbb{R\rightarrow R}$ are
continuous for $\delta _{i}=\alpha _{i}=1$,\ we have%
\begin{equation*}
\left\vert \psi _{i}\left( t_{1}\right) -\psi _{i}\left( t_{2}\right)
\right\vert \leqslant \delta _{i}\left\vert t_{1}-t_{2}\right\vert ^{\alpha
_{i}},
\end{equation*}%
for any $t_{1},t_{2}\in \mathbb{R}_{+}$. The conditions $\left( i\right) $
and $\left( ii\right) $ hold.

Now, let%
\begin{eqnarray*}
\left\vert f_{1}\left( t,x,y,z,p\right) -f_{1}\left( t,u,v,w,\rho \right)
\right\vert &=&\left\vert \frac{x+y+z}{3t^{2}+3}+p-\left( \frac{u+v+w}{%
3t^{2}+3}+\rho \right) \right\vert \\
&\leqslant &\frac{1}{3t^{2}+3}\left[ \left\vert x-u\right\vert +\left\vert
y-v\right\vert +\left\vert z-w\right\vert \right] +\left\vert p-\rho
\right\vert \\
&\leqslant &\frac{3}{3t^{2}+3}\max \left\{ \left\vert x-u\right\vert
,\left\vert y-v\right\vert ,\left\vert z-w\right\vert \right\} +\left\vert
p-\rho \right\vert \\
&\leqslant &\frac{1}{t^{2}+1}\max \left\{ \left\vert x-u\right\vert
,\left\vert y-v\right\vert ,\left\vert z-w\right\vert \right\} +\left\vert
p-\rho \right\vert \\
&=&\varphi _{1}\left( \max \left\{ \left\vert x-u\right\vert ,\left\vert
y-v\right\vert ,\left\vert z-w\right\vert \right\} \right) +\Phi \left(
\left\vert p-\rho \right\vert \right) .
\end{eqnarray*}%
Similarly, we prove that%
\begin{equation*}
\left\vert f_{2}\left( t,x,y,z,p\right) -f_{2}\left( t,u,v,w,\rho \right)
\right\vert \leqslant \varphi _{2}\left( \max \left\{ \left\vert
x-u\right\vert ,\left\vert y-v\right\vert ,\left\vert z-w\right\vert
\right\} \right) +\Phi \left( \left\vert p-\rho \right\vert \right)
\end{equation*}%
and%
\begin{equation*}
\left\vert f_{3}\left( t,x,y,z,p\right) -f_{3}\left( t,u,v,w,\rho \right)
\right\vert \leqslant \varphi _{3}\left( \max \left\{ \left\vert
x-u\right\vert ,\left\vert y-v\right\vert ,\left\vert z-w\right\vert
\right\} \right) +\Phi \left( \left\vert p-\rho \right\vert \right) .
\end{equation*}%
Then, $\left( iii\right) $ also holds.

In further $\left( iv\right) $ is valid. Indeed,%
\begin{equation*}
M_{i}=\sup \left\vert \left\{ f_{i}\left( t,0,0,0,0\right) :t\in \mathbb{R}%
_{+}\right\} \right\vert =0,i=1,2,3.
\end{equation*}%
Let us verify the last condition $\left( v\right) .$ First, note that%
\begin{eqnarray*}
&&\left\vert h_{1}\left( t,s,x,y,z\right) -h_{1}\left( t,s,u,v,w\right)
\right\vert \\
&=&\left\vert \frac{xs\left\vert \sin y\right\vert \left\vert \cos
z\right\vert }{e^{t}\left( 1+x^{2}\right) \left( 1+\sin ^{2}y\right) \left(
1+\cos ^{2}z\right) }-\frac{us\left\vert \sin v\right\vert \left\vert \cos
w\right\vert }{e^{t}\left( 1+u^{2}\right) \left( 1+\sin ^{2}v\right) \left(
1+\cos ^{2}w\right) }\right\vert \\
&\leqslant &\left\vert \frac{x}{1+x^{2}}\frac{s}{e^{t}}-\frac{u}{1+u^{2}}%
\frac{s}{e^{t}}\right\vert \leqslant \frac{1}{2}\frac{s}{e^{t}}+\frac{1}{2}%
\frac{s}{e^{t}} \\
&=&\frac{s}{e^{t}}.
\end{eqnarray*}%
Hence,%
\begin{eqnarray*}
&&\lim_{t\rightarrow \infty }\int_{0}^{t}\left\vert h_{1}\left( t,s,x\left(
\eta \left( s\right) \right) ,y\left( \eta \left( s\right) \right) ,z\left(
\eta \left( s\right) \right) \right) -h_{1}\left( t,s,u\left( \eta \left(
s\right) \right) ,v\left( \eta \left( s\right) \right) ,w\left( \eta \left(
s\right) \right) \right) \right\vert ds \\
&\leqslant &\lim_{t\rightarrow \infty }\int_{0}^{t}\frac{s}{e^{t}}ds=0.
\end{eqnarray*}%
In addition,%
\begin{eqnarray*}
&&\left\vert h_{2}\left( t,s,x,y,z\right) -h_{2}\left( t,s,u,v,w\right)
\right\vert \\
&=&\left\vert \frac{e^{s}\left( y^{2}\right) \left( 1+\cos ^{2}x\right)
\left( 1+\sin ^{2}z\right) }{e^{t^{2}}\left( 1+y^{2}\right) \left( 1+\sin
^{2}x\right) \left( 1+\cos ^{2}z\right) }-\frac{e^{s}\left( v^{2}\right)
\left( 1+\cos ^{2}u\right) \left( 1+\sin ^{2}w\right) }{e^{t^{2}}\left(
1+v^{2}\right) \left( 1+\sin ^{2}u\right) \left( 1+\cos ^{2}w\right) }%
\right\vert \\
&\leqslant &\left\vert \frac{y^{2}}{1+y^{2}}\frac{e^{s}}{e^{t^{2}}}-\frac{%
v^{2}}{1+v^{2}}\frac{e^{s}}{e^{t^{2}}}\right\vert \leqslant 2\frac{e^{s}}{%
e^{t^{2}}}.
\end{eqnarray*}%
Thus,%
\begin{eqnarray*}
&&\lim_{t\rightarrow \infty }\int_{0}^{t}\left\vert h_{2}\left( t,s,x\left(
\eta \left( s\right) \right) ,y\left( \eta \left( s\right) \right) ,z\left(
\eta \left( s\right) \right) \right) -h_{2}\left( t,s,u\left( \eta \left(
s\right) \right) ,v\left( \eta \left( s\right) \right) ,w\left( \eta \left(
s\right) \right) \right) \right\vert ds \\
&\leqslant &\lim_{t\rightarrow \infty }\int_{0}^{t}2\frac{e^{s}}{e^{t^{2}}}%
ds=0.
\end{eqnarray*}%
Moreover,%
\begin{eqnarray*}
&&\left\vert h_{3}\left( t,s,x,y,z\right) -h_{3}\left( t,s,u,v,w\right)
\right\vert \\
&=&\left\vert \frac{s^{2}\left\vert \cos z\right\vert +\sqrt{e^{s}\left(
1+z^{2}\right) \left( 1+\sin ^{2}y\right) \left( 1+\cos ^{2}x\right) }}{%
e^{t}\left( 1+z^{2}\right) \left( 1+\sin ^{2}y\right) \left( 1+\cos
^{2}x\right) }-\frac{s^{2}\left\vert \cos w\right\vert +\sqrt{e^{s}\left(
1+w^{2}\right) \left( 1+\sin ^{2}v\right) \left( 1+\cos ^{2}u\right) }}{%
e^{t}\left( 1+w^{2}\right) \left( 1+\sin ^{2}v\right) \left( 1+\cos
^{2}u\right) }\right\vert \\
&\leqslant &\left\vert \frac{s^{2}}{e^{t}}\left( \cos z-\cos w\right)
\right\vert \leqslant \frac{s^{2}}{e^{t}}.
\end{eqnarray*}%
Then,%
\begin{eqnarray*}
&&\lim_{t\rightarrow \infty }\int_{0}^{t}\left\vert h_{2}\left( t,s,x\left(
\eta \left( s\right) \right) ,y\left( \eta \left( s\right) \right) ,z\left(
\eta \left( s\right) \right) \right) -h_{2}\left( t,s,u\left( \eta \left(
s\right) \right) ,v\left( \eta \left( s\right) \right) ,w\left( \eta \left(
s\right) \right) \right) \right\vert ds \\
&\leqslant &\lim_{t\rightarrow \infty }\int_{0}^{t}\frac{s^{2}}{e^{t}}ds=0.
\end{eqnarray*}%
Furthermore, \ for any $x,y,z\in BC\left( \mathbb{R}_{+}\right) \times
BC\left( \mathbb{R}_{+}\right) \times BC\left( \mathbb{R}_{+}\right) ,$%
\begin{equation*}
\sup \left\{ \left\vert \int_{0}^{t}h_{i}\left( t,s,x\left( \eta \left(
s\right) \right) ,y\left( \eta \left( s\right) \right) ,z\left( \eta \left(
s\right) \right) \right) ds\right\vert ,\text{ }t,s\in \mathbb{R}%
_{+}\right\} <D.
\end{equation*}%
It is easy to see that for an $r_{0}>0,$ we have%
\begin{equation*}
\varphi \left( r_{0}\right) +\frac{1}{2}+\Phi \left( D\right) \leqslant
r_{0},
\end{equation*}%
holds and the condition $\left( v\right) $ is valid.

Finally, the system has at least one solution in $BC\left( \mathbb{R}%
_{+}\right) \times BC\left( \mathbb{R}_{+}\right) \times BC\left( \mathbb{R}%
_{+}\right) .$
\end{example}

\bigskip


\end{document}